\documentclass[preprint,12pt,3p]{elsarticle}

\usepackage{amsmath,amsfonts,amsthm,amssymb, graphicx,color}
\usepackage{mathrsfs}
\usepackage[unicode,breaklinks=true,colorlinks=true]{hyperref}
\usepackage[dvipsnames]{xcolor}
\usepackage{graphicx}

\newtheorem{theorem}{Theorem}[section]

\theoremstyle{definition}
\newtheorem{definition}[theorem]{Definition}

\theoremstyle{remark}
\newtheorem{remark}[theorem]{Remark}

\numberwithin{equation}{section}

\journal{XXX}

\begin{document}

\begin{frontmatter}

\title{Liouville's theorem for the generalized harmonic function\\ [2ex] \begin{footnotesize} \textit{This paper is dedicated to our advisors} \end{footnotesize} }
\tnotetext[label0]{This work was supported by National Natural Science Foundation of China (Grant No.11471309, 11771423), Putian University Research Projects (2016095) and Teaching Reform Project in Putian University (No. JG201524, No. JG201628).}
\author[label1]{Weihua Wang\corref{cor1}}
\address[label1]{School of Mathematical Sciences, University of Chinese Academy of Sciences, Beijing, P.R.China 100049}
\address[label2]{School of Mathematics,~Putian University, Putian 351100, P.R. China\fnref{label4}}
\cortext[cor1]{Corresponding author}

\ead{wangvh@163.com}

\author[label2]{Qihua Ruan}


\begin{abstract}
In this paper, we give a more physical proof of Liouville's theorem for a class generalized harmonic functions by the method of parabolic equation.
\end{abstract}




\end{frontmatter}


\section*{}
Liouville's theorem for harmonic functions is based on the mean value property (\cite{Ev}), which has a self-evident advantage  that the theorem is derived via
Harnack's inequality under a weaker assumption that the function is bounded below (\cite{GT}). Miyazaki \cite{Mi} give a proof of Liouville's theorem for harmonic functions by the method of heat kernels, which reflects the physical essence that there is no more heat exchange after the system of heat diffusion
reaches thermal equilibrium. Inspired by Miyazaki's method \cite{Mi},  we  present a proof of
Liouville's theorem  for a class generalized harmonic function by the method of parabolic equation. In fact, we generalize the result of Theorem 1 in \cite{Mi}.
\begin{definition}\label{def1}
We say $u(x)$ is a generalized harmonic function in $\Omega $ in $\mathbb{R}^3$ if there exists  a constant vector $\boldsymbol{c}$ in $\mathbb{R}^d$ such that  $u$ satisfy $-\triangle u + \boldsymbol{c}\cdot\nabla u=0$ in the distributional sense:
\begin{equation}\label{eq1}
  \int_{\Omega}u(x)(\triangle \varphi(x) + \boldsymbol{c}\cdot\nabla \varphi(x))dx=0
\end{equation}
 for all  $\varphi\in C_0^{\infty}(\Omega)$, where $\triangle= \sum_{j=1}^d \frac{\partial^2}{\partial x_j^2}$, $\nabla=(\frac{\partial}{\partial x_1},\cdots, \frac{\partial}{\partial x_d} )$.
\end{definition}

\begin{remark}
 In definition\ref{def1}, u is a classic harmonic function in the distributional sense when $\boldsymbol{c}\equiv0$.
\end{remark}

\begin{theorem}[Liouville's theorem]\label{thm1}
 Let the generalized harmonic function $u$ is  bounded continuous  in $\mathbb{R}^d$, then $u$ must be a constant.
\end{theorem}

\begin{proof}
Let
\begin{equation*}{
  K(x,t)=
  \left\{
  \begin{array}{ll}
  (4\pi t)^{-\frac{d}{2}}\exp\left(-\frac{(x-\boldsymbol{c}t)^2}{4t}\right), &(s\in \mathbb{R}^d, t>0)\\
  0 &(s\in\mathbb{R}^d, t<0),
  \end{array}
\right.}
\end{equation*}
\normalsize then for any non-zero $t\in\mathbb{R}$, $K(x,t)$ is a function in the Schwartz space $\mathscr{S}(\mathbb{R}^d)$ and satisfied the  parabolic equation
\begin{equation*}
  \partial_t K(x,t)-\triangle K(x,t)+\boldsymbol{c}\cdot\nabla K(x,t)=0
\end{equation*}
and
\begin{equation*}
  \int_{\mathbb{R}^d}K(x,t)dx=1.
\end{equation*}
We define the function
\begin{eqnarray*}
  v(x,t) &=& K(\cdot, t)\ast u(x)\\
         &=& \int_{\mathbb{R}^d}K(x-y,t)u(y)dy,
\end{eqnarray*}
which is in $C^{\infty}(\mathbb{R}^d\times\mathbb{R}_+)$. We notice the equation \eqref{eq1} also holds for all $\varphi$ in  $\mathscr{S}(\mathbb{R}^d)$ by the bound of $u$, since $C_0^{\infty}(\mathbb{R}^d)$ is dense in  $\mathscr{S}(\mathbb{R}^d)$. By the properties of $K(x,t)$ and assumption on $u$ we have

\begin{eqnarray*}
  & &\partial_t  v(x,t)\\
  &=& \int_{\mathbb{R}^d}\partial_t K(x-y,t)u(y)dy \\
   &=& \int_{\mathbb{R}^d}\{(\triangle_x -\boldsymbol{c}\cdot\nabla_x)K(x-y,t)\}u(y)dy \\
   &=& \int_{\mathbb{R}^d}\{(\triangle_y +\boldsymbol{c}\cdot\nabla_y)K(x-y,t)\} u(y)dy\\
   &=&0.
\end{eqnarray*}
Hence $v(x,t)$ is independent of $t$. Since $\lim_{t\to 0^+}v(x,t)=u(x)$, we have
$$v(x,t)=u(x),$$ which implies that $u$ is a $C^{\infty}$ function. Differentiating in $x_j$ gives

\begin{equation*}
  \partial_{x_j}u(x) =  \int_{\mathbb{R}^d}\partial_{x_j}K(x-y,t)u(y)dy,
\end{equation*}
hence, \begin{eqnarray*}
         & &|\partial_{x_j}u(x)| \\
         &\leq& \|\partial_{x_j}K(x,t)\|_{L^1(\mathbb{R}^d)} \|u\|_{L^{\infty}(\mathbb{R}^d)}\\
                               &=&\int_{\mathbb{R}^d}\left|(4\pi t)^{-\frac{d}{2}}\exp\left(-\frac{(x-\boldsymbol{c}t)^2}{4t}\right)\frac{(x_i-\boldsymbol{c}_it)}{2t}\right|dx \|u\|_{L^{\infty}(\mathbb{R}^d)}\\
                               &=& t^{-\frac{1}{2}}\int_{\mathbb{R}^d}\left|(4\pi )^{-\frac{d}{2}}\exp(-y^2)y_i\right|dy\|u\|_{L^{\infty}(\mathbb{R}^d)}.
       \end{eqnarray*}
   Here, the replacement $$y=\frac{x-\boldsymbol{c}t}{2t^{\frac{1}{2}}}$$ is used  in the last equation of the above equations.\\
    Let $t\to +\infty$, and we obtain $\partial_{x_j}u(x) =0$ for all $x\in \mathbb{R}^d$ and $1\leq j\leq d$. Therefore $u(x)$ is a constant.
\end{proof}
\begin{remark}
 In theorem \ref{thm1}, we obtain immediately the result of theorem 1 in \cite{Mi} just by letting $\boldsymbol{c}=0$.
\end{remark}

\section*{Acknowledgments}
Two authors would like to thank the referees for their valuable suggestions.

\section*{References}
\bibliographystyle{elsarticle-num}


\end{document}